\documentclass{article}
\usepackage[english]{babel}
\usepackage[cp850]{inputenc}
\usepackage{amscd}
\usepackage{amssymb}
\usepackage{amsmath}
\usepackage{amsthm}
\usepackage{graphics}

\newtheorem{exa}{Example}
\newtheorem{lem}{Lemma}
\newtheorem{prop}{Proposition}

\newcommand{\al}{\alpha}
\newcommand{\be}{\beta}
\newcommand{\ga}{\gamma}
\newcommand{\de}{\delta}
\newcommand{\la}{\lambda}
\newcommand{\ro}{\rho}

\newcommand{\rt}{\sqrt{3}}

\newcommand{\bal}{\bar\alpha}
\newcommand{\bbe}{\bar\beta}

\title{Some results on the Dunkl--Williams constant}
\author{Javier Alonso and Pedro Mart\'{\i}n}
\date{\today}

\begin{document}

\maketitle


\begin{abstract}
This paper presents a compilation of various formulas for calculating the Dunkl-Williams constant $DW(X)$ of a real normed linear space. The constant $DW_B(X)$ related to Birkhoff orthogonality is also considered. The value of $DW(X)$ is calculated for several two-dimensional spaces. In particular, it is shown that the Dunk-Williams constant for $\ell_2-\ell_1$ is equal to $2\sqrt{2}$, and that it is equal to $8(2-\rt)$ for the two dimensional normed linear space whose unit sphere is a dodecahedron.
\end{abstract}

\vspace*{.5cm}

Let $X$ be a real normed linear space with unit sphere $S_X$. The Dunkl--Williams constant
$DW(X)$ was defined by A. Jim\'{e}nez-Melado, E. Llorens-Fuster and E.M. Mazcu\~{n}\'{a}n-Navarro $\cite{JLM}$ as
$$
DW(X)=\sup\left\{\frac{\|x\|+\|y\|}{\|x-y\|}\left\|\frac{x}{\|x\|}-\frac{y}{\|y\|}\right\|:\, x,y\in X\setminus\{0\},\,x\neq y\right\}.
$$

The origin of this constant lies in a paper by C.F. Dunkl and K.S. Williams $\cite{DW}$, who proved that the inequality
\begin{equation}\label{D_W inequality}
\left\|\frac{x}{\|x\|}-\frac{y}{\|y\|}\right\|\leq \frac{4\|x-y\|}{\|x\|+\|y\|}
\end{equation}
holds for all $x,y\in X\setminus{\{0\}}$. Thus, inequality (\ref{D_W inequality}) says that $DW(X)\leq 4$. Moreover (take $y=-x$), we have that $DW(X)\geq 2$. The same year,  Kirk and Smiley $\cite{DW2}$ proved that $DW(X)=2$ characterizes inner product spaces. This characterization can be improved in the following sense: Y. Fu, H. Xie and  Y. Li \cite{YHY} defined the constant
$$
DW_B(X)=\sup\left\{\frac{\|x\|+\|y\|}{\|x-y\|}\left\|\frac{x}{\|x\|}-\frac{y}{\|y\|}\right\|:\, x,y\in X\setminus\{0\},\,x\perp_B y\right\},
$$
where $x\perp_B y$ means that $x$ is Birkhoff orthogonal to $y$ (i.e. $\|x+\lambda y\|\geq\|x\|$ for all $\lambda\in \mathbb{R}$) and observed  \cite[Corollary 2.10]{YHY} that $DW_B(X)=2$ also characterizes inner product spaces. For the proof, they refer to the book by D. Amir \cite{Am}, who in turn cites the doctoral dissertation \cite{Al1} as the source of the result (which was also published in \cite{Al2}). Since both references are difficult to obtain, the proof is reproduced  here, in Proposition \ref{Characterization}. But first, let us note that $DX(X)$, as well as $DW_B(X)$, can be defined in different ways.

\begin{prop}\label{Equivalences}
Let $X$ be a real normed linear space, and let us consider
\begin{align*}
DW_1(X) & :=\sup\left\{\frac{(\lambda+\mu)\|u+v\|}{\|\lambda u+\mu v\|}: u,v\in S_X,\ 0<\lambda<\mu\right\},\\[1ex]
DW_2(X) & := \sup\left\{\frac{\|u+v\|}{\|\gamma u+(1-\gamma) v\|}: u,v\in S_X,\ 0<\gamma<\frac{1}{2}\right\},\\[1ex]
DW_3(X) & := \sup\left\{\frac{\|u+v\|}{\inf\limits_{0<\gamma<\frac{1}{2}}\|\gamma u+(1-\gamma) v\|}: u,v\in S_X,\ u+v\neq0  \right\},\\[1ex]
DW_4(X) & := \sup\left\{\frac{2\|u+v\|}{\|u+v+\delta (u-v)\|}: u,v\in S_X,\ 0<\delta<1\right\},\\[1ex]
DW_5(X) & := \sup\left\{\frac{(1+t)\|u+v\|}{\|u+tv\|}: u,v\in S_X,\ 0<t<1\right\}.
\end{align*}
Then, $DW(X)=DW_i(X)$, $i=1,\ldots, 5$. 
\end{prop}

\begin{proof}
$DW(X)=DW_1(X)$:
Let $x,y\in X\setminus\{0\}$, $x\neq y$, and assume, without loss of generality, that $\|x\|<\|y\|$. Let $u=x/\|x\|$, $v=-y/\|y\|$, $\lambda=1/\|y\|$ and $\mu=1/\|x\|$. Then, $u,v\in S_X$, $0<\lambda<\mu$ and
$$
\frac{\|x\|+\|y\|}{\|x-y\|}\left\|\frac{x}{\|x\|}-\frac{y}{\|y\|}\right\| = \frac{(\lambda+\mu)\|u+v\|}{\|\lambda u+\mu v\|} \leq DW_1(X).
$$
So we get that $DW(X)\leq DW_1(X)$. On the other hand, let $u,v\in S_X$ and $0<\lambda<\mu$. By taking $x=\lambda u$ and $y=-\mu v$, we have that $x,y\in X\setminus\{0\}$, $x\neq y$, and
$$
\frac{(\lambda+\mu)\|u+v\|}{\|\lambda u+\mu v\|}=\frac{\|x\|+\|y\|}{\|x-y\|}\left\|\frac{x}{\|x\|}-\frac{y}{\|y\|}\right\|\leq DW(X),
$$
and we get that $DW_1(X)\leq DW(X)$.

\smallskip

$DW_1(X)=DW_2(X)$: Let $u,v\in S_X$ and $0<\lambda<\mu$. Take $\gamma=\lambda/(\lambda+\mu)$. Then, $0<\gamma<1/2$, and
$$
\frac{(\lambda+\mu)\|u+v\|}{\|\lambda u+\mu v\|}=\frac{\|u+v\|}{\|\gamma u+(1-\gamma)v\|}\leq DW_2(X).
$$
Therefore, $DW_1(X)\leq DW_2(X)$. On the other hand, let $u,v\in S_X$, and $0<\gamma<1/2$. Let $\lambda=\gamma/(1-\gamma)$ and $\mu=1$. Then $0<\lambda<\mu$, and
$$
\frac{\|u+v\|}{\|\gamma u+(1-\gamma)v\|}=\frac{(\lambda+\mu)\|u+v\|}{\|\lambda u+\mu v\|}\leq DW_1(X),
$$
and we get that $DW_2(X)\leq DW_1(X)$.

\smallskip

$DW_2(X)=DW_3(X)$: Let's first see that if $u+v\neq0$, then $\inf_{0<\gamma<\frac{1}{2}}\|\gamma u+(1-\gamma) v\|>0$. If $u=v$, then $\inf_{0<\gamma<\frac{1}{2}}\|\gamma u+(1-\gamma) v\|=1$. Assume, on the contrary, that $u\neq v$. Then, the points $u$, $-u$ and $v$ form a non-degenerate triangle, which implies that the distance from the origin to the line through $u$ and $v$ is strictly positive, and then  $\inf_{0<\gamma<\frac{1}{2}}\|\gamma u+(1-\gamma) v\|>0$.

Now, let $\bar u,\bar v\in S_X$, $\bar u+\bar v\neq 0$,  and $\bar \gamma\in (0,\frac{1}{2})$. Then,
$$
\frac{\|\bar u+\bar v\|}{\|\bar\gamma \bar u+(1-\bar\gamma)\bar v\|}\leq \frac{\|\bar u+\bar v\|}{\inf\limits_{0<\gamma<\frac{1}{2}}\|\gamma \bar u+(1-\gamma) \bar v\|}\leq DW_3(X).
$$
Taking the supreme on the left, we obtain $DW_2(X)\leq DW_3(X)$. On the other hand, let $\bar u,\bar v\in S_X$, $\bar u+\bar v\neq 0$. For any $\varepsilon>0$, there exists $\gamma(\bar u,\bar v,\varepsilon)\in (0,\frac{1}{2})$ such that
$$
\left\|\gamma(\bar u,\bar v,\varepsilon) \bar u+\big(1-\gamma(\bar u,\bar v,\varepsilon)\big)\bar v\right\|<\inf\limits_{0<\gamma<\frac{1}{2}}\|\gamma\bar u+(1-\gamma)\bar v\|+\varepsilon.
$$
Then,
$$
DW_2(X)\geq\frac{\|\bar u+\bar v\|}{\left\|\gamma(\bar u,\bar v,\varepsilon) \bar u+\big(1-\gamma(\bar u,\bar v,\varepsilon)\big)\bar v\right\|}>
\frac{\|\bar u+\bar v\|}{\inf\limits_{0<\gamma<\frac{1}{2}}\|\gamma \bar u+(1-\gamma)\bar v\|+\varepsilon}.
$$
Letting epsilon tend to zero, and taking supreme on the right, we get $DW_2(X)\geq DW_3(X)$.

\smallskip

$DW_2(X)=DW_4(X)$: Let $u,v\in S_X$ and $0<\gamma<1/2$. Take $\delta=1-2\gamma$. Then, $0<\delta<1$, and
$$
\frac{\|u+v\|}{\|\gamma u+(1-\gamma)v\|}=\frac{2\|v+u\|}{\|v+u+\delta(v-u)\|}\leq DW_4(X),
$$
and we get $DW_2(X)\leq DW_4(X)$. On the other hand, let  $u,v\in S_X$ and $0<\delta<1$. By taking $\gamma=(1-\delta)/2$, we have that $0<\gamma<1/2$, and
$$
\frac{2\|u+v\|}{\|u+v+\delta(u-v)\|}=\frac{\|v+u\|}{\|\gamma v+(1-\gamma)u\|}\leq DW_2(X),
$$
and we get that $DW_4(X)\leq DW_2(X)$.

\smallskip

$DW_4(X)=DW_5(X)$: Let $u,v\in S_X$ and $0<\delta<1$. Take $t=(1-\delta)/(1+\delta)$. Then, $0<t<1$, and
$$
\frac{2\|u+v\|}{\|u+v+\delta(u-v)\|}=\frac{(1+t)\|u+v\|}{\|u+t v\|}\leq DW_5(X),
$$
and we get that $DW_4(X)\leq DW_5(X)$. In the same way, it is prove that $DW_5(X)\leq DW_4(X)$.
\end{proof}

The following proposition follows easily from the proof of Proposition~\ref{Equivalences} by taking into account that Birkhoff orthogonality is homogeneous, but (in general) not symmetric (see, for example, \cite{AlMaWu1} and \cite{AlMaWu2}).

\begin{prop}\label{Equivalences2}
Let $X$ be a real normed linear space, and let us consider
\begin{align*}
DW_{B1}(X) & :=\sup\left\{\frac{(\lambda+\mu)\|u+v\|}{\|\lambda u+\mu v\|}: u,v\in S_X, u\perp_B v,\ \lambda>0,\ \mu>0\right\},\\[1ex]
DW_{B2}(X) & := \sup\left\{\frac{\|u+v\|}{\|\gamma u+(1-\gamma) v\|}: u,v\in S_X, u\perp_B v,\ 0<\gamma<1\right\},\\[1ex]
DW_{B3}(X) & := \sup\left\{\frac{\|u+v\|}{\inf\limits_{0<\gamma<1}\|\gamma u+(1-\gamma) v\|}: u,v\in S_X,\ u\perp_B v  \right\},\\[1ex]
DW_{B4}(X) & := \sup\left\{\frac{2\|u+v\|}{\|u+v+\delta (u-v)\|}: u,v\in S_X, u\perp_B v,\ -1<\delta<1\right\},\\[1ex]
DW_{B5}(X) & := \sup\left\{\frac{(1+t)\|u+v\|}{\|u+tv\|}: u,v\in S_X,\ u\perp_B v,t>0\right\}.
\end{align*}
Then, $DW_B(X)=DW_{Bi}(X)$, $i=1,\ldots, 5$.
\end{prop}

\begin{prop}\label{Characterization}
A real normed linear space $X$ is an inner product space if, and only if, the property
\begin{equation}\label{D_W inequality2}
x,y\in X\setminus{\{0\}},\ x\perp_B y\quad \Rightarrow\quad \left\|\frac{x}{\|x\|}-\frac{y}{\|y\|}\right\|\leq \frac{2\|x-y\|}{\|x\|+\|y\|}
\end{equation}
holds. That is, $X$ is an inner product space if and only if $DW_B(X)=2$.
\end{prop}

\begin{proof}
It is clear that if $X$ is an inner product space, then the property (\ref{D_W inequality2}) holds. Conversely, assume that $DW_B(X)=2$. By Proposition \ref{Equivalences2}, we have that if $u,v\in S_X$, $u\perp_B v$ and $-1<\delta<1$, then $\|u+v\|\leq \|u+v+\delta(u-v)\|$. Since for any $x,y\in X$, the function $\delta\in \mathbb{R}\mapsto \|x+\delta y\|$ is convex, we have that $\|u+v\|\leq \|u+v+\delta(u-v)\|$ for all $\delta\in \mathbb{R}$. This means that the property
$$
u,v\in S_X,\ u\perp_B v\quad \Rightarrow \quad u+v\perp_B u-v
$$
holds. It was probed by M. Baronti \cite{Ba} that this property characterizes inner product spaces.
\end{proof}

\section*{Some examples}

Next, we will calculate $DW(X)$ for some spaces that frequently appear in the literature. Given the tedious nature of these computations, the use of a computer algebra system (e.g., Mathematica or Maxima) is recommended for verification.

To simplify the calculations, in Examples \ref{hexagon} and \ref{l2-l1} we will consider $DW(X)$ in the following form (recall Proposition \ref{Equivalences}):
$$
DW(X)=\sup\{dw(u,v,t):\,u,v\in S_X,\ 0<t<1\},
$$
where
$$
dw(u,v,t)=\frac{(1+t)\|u-v\|}{\|u-tv\|}.
$$
In Example \ref{dodecahedron} we will use the identity $DW(X)=DW_3(X)$.


The result in Example \ref{hexagon} is known. In \cite{Mi1}, H. Mizuguchi defined the constant
$$
IB(X)=\inf\left\{\frac{\inf_{\lambda\in \mathbb{R}}\|x+\lambda y\|}{\|x\|}: x,y\in X\setminus\{0\},\ \|x+y\|=\|x-y\|\right\},
$$
and probed that for any space $X$, $IB(X)DW(X)=2$. Later,  Mizuguchi\;\cite{Mi2}, after some cumbersome lemmas, showed that if $X$ is a Radon plane (as is the case in Example \ref{hexagon}) then $IB(X)\geq 8/9$, with the equality if and only if the unit sphere is affine to a regular hexagon. This gives $DW(X)=9/4$ in Example~\ref{hexagon}. What we do here is a direct calculation of this value, confirming the validity of Mizuguchi's result\footnote{We must keep in mind that there is a preprint in the web where a different value is given.}.

\begin{exa}\label{hexagon}
Let $X$ be the space $\ell_\infty-\ell_1$, i.e., $\mathbb{R}^2$ endowed with the norm
$$
\|(x_1,x_2)\|=
\begin{cases}
\;\max\{|x_1|,|x_2|\} & \text{if\ \ }
x_1x_2\geq 0,\\[.5ex]
\;|x_1|+|x_2| & \text{if\ \ } x_1x_2\leq 0,
 \end{cases}
$$
whose unit sphere is affine to a regular hexagon. Then $DW(X)=9/4$. Moreover, $DW(X)=dw(u,v,t)$ only when $t=1/2$ and $u$ and $v$
are the points $(1,1/2)$ and  $(0,1)$, or any other couple of points with a
similar position in the unit sphere (that is affine to a regular hexagon).
\end{exa}

{\it Proof.} To compute $DW(X)$ we shall consider several cases according to
the position of $u$ and $v$ over $S_X$. Let $0<t<1$.

\medskip

{\sl Case}\/ 1. Assume that $u=(1,\al)$ and $v=(1,\be)$ with $0\leq\al\leq 1$
and $0\leq\be\leq 1$. Then $\|u-v\|=\|(0,\al-\be)\|=|\al-\be|$ and
$$
\|u-tv\|=\|(1-t,\al-t\be)\|=
\begin{cases}
\max\{1-t,\al-t\be\} & \text{if } \al\geq t\be\\[.5ex]
1+t+t\be-\al & \text{if } \al\leq t\be
\end{cases}
$$

1.1. Assume that $\al\leq\be$. If $\al\geq t\be$,
then $\al-t\be\leq\al+1-t-\be\leq1-t$, and
$$
dw(u,v,t)=\frac{(1+t)(\be-\al)}{1-t}\leq\al+\be\leq 2<\frac{9}{4}.
$$
On the contrary, if $\al\leq t\be$, then
$$
dw(u,v,t)=\frac{(1+t)(\be-\al)}{1-t+t\be-\al}\leq\frac{2(\be-\al)}{1-t+t\be-\al}\leq2<\frac{9}{4}.
$$

1.2. Assume that $\al\geq\be$. Since $\al-t\be\geq
t(\al-\be)\geq 0$, we have $\|u-tv\|=\max\{1-t,\al-t\be\}$. If
$1-t\geq\al-t\be$, then
$$
dw(u,v,t)=\frac{(1+t)(\al-\be)}{1-t}\leq 2-\al-\be\leq 2<\frac{9}{4},
$$
On the contrary, if $1-t\leq\al-t\be$, then
$$
dw(u,v,t)=\frac{(1+t)(\al-\be)}{\al-t\be}\leq 2<\frac{9}{4}.
$$

\medskip

{\sl Case}\/ 2. Assume that $u=(1,\al)$ and $v=(\be,1)$ with $0\leq\al\leq 1$
and $0\leq\be< 1$. Then $\|u-v\|=\|(1-\be,\al-1)\|=|1-\be|+|\al-1|=2-\al-\be$.
Moreover, since $1-t\be>0$,
$$
\|u-tv\|=\|(1-t\be,\al-t)\|=
\begin{cases}
\max\{1-t\be,\al-t\} & \text{ if } \al\geq t,\\[.5ex]
1-t\be+t-\al & \text{ if } \al\leq t.
\end{cases}
$$

2.1. Assume that $\alpha\geq t>0$. Then $1-t\be\geq\al-t\be > \al-t\geq0$, and
$$
dw(u,v,t)=\frac{(1+t)(2-\al-\be)}{1-t\be}=\frac{9}{4}-\frac{f(\al,\be,t)}{4(1-t\be)}.
$$
where $f(\al,\be,t):=(4\al-5\be-8)t+4\al+4\be+1$. Simple calculations show that
$$
f(\al,\be,t)=\Big(\frac{\al-t}{\al}\Big)f(\al,\be,0)+\frac{t}{\al}\Big[(1-\be)f(\al,0,\al)+\be f(\al,1,\al)\Big].
$$
Since
\begin{align*}
f(\al,\be,0) & = 4\al+4\be+1 >0,\\
f(\al,0,\al) & = (2\al-1)^2\geq 0,\\
f(\al,1,\al) & = (1-\al)(5-4\al)>0,
\end{align*}
we have that $f(\al,\be,t)\geq 0$. Moreover, $f(\al,\be,t)=0$ if and only if $t=\al$, $f(\al,0,\al)=0$ and $\be=0$, i.e., $u=(1,1/2)$, $v=(0,1)$, and $t=1/2$.

2.2. Assume that $\al\leq t$. Then
$$
dw(u,v,t)=\frac{(1+t)(2-\al-\be)}{1-t\be+t-\al}=\frac{9}{4}-\frac{g(\al,\be,t)}{4(1-t\be+t-\al)},
$$
where $g(\al,\be,t):=(4\al-5\be+1)t-5\al+4\be+1$. Since,
$$
g(\al,\be,t)=\Big(\frac{t-\al}{t}\Big)g(0,\be,t)+\frac{\al}{t}\big[(1-\be)g(t,0,t)+\be g(t,1,t)\big].
$$
with
\begin{align*}
g(0,\be,t) & = 4\be(1-t)+t(1-\be)+1>0,\\
g(t,0,t) & = (2t-1)^2\geq 0,\\
g(t,1,t) & = (1-t)(5-4t)>0,
\end{align*}
we have that $g(\al,\be,t)\geq 0$. Moreover, $g(\al,\be,t)=0$ if and only if $t=\al$, $g(t,0,t)=0$ and $\be=0$, i.e., again $u=(1,1/2)$, $v=(0,1)$, and $t=1/2$.

\medskip

{\sl Case}\/ 3. Assume that $u=(1,\al)$ and $v=(\be-1,\be)$ with $0\leq\al\leq
1$ and $0\leq\be< 1$. Then $u-v=(2-\be,\al-\be)$, and
$$
\|u-v\|=\left.
\begin{cases}
\max\{2-\be,\al-\be\} & \text{ if }\al\geq \be\\[.5ex]
2-\be+\be-\al & \text{ if }\al\leq \be
\end{cases}
\right\} =
\begin{cases}
2-\be & \text{ if }\al\geq \be,\\[.5ex]
2-\al & \text{ if }\al\leq \be.
\end{cases}
$$
Moreover, $u-tv=(1+t-t\be,\al-t\be)$, and
$$
\|u-tv\|=\left.
\begin{cases}
\max\{1+t-t\be,\al-t\be\} & \text{ if }\al\geq t\be\\[.5ex]
1+t-t\be+t\be-\al & \text{ if }\al\leq t\be
\end{cases}
\right\} =
\begin{cases}
1+t-t\be & \text{ if }\al\geq t\be,\\[.5ex]
1+t-\al & \text{ if }\al\leq t\be.
\end{cases}
$$

3.1. Assume that $\al\geq\be$. Then $\al\geq t\be$, and
$$
dw(u,v,t)=\frac{(1+t)(2-\be)}{1+t-t\be}\leq 2<\frac{9}{4}.
$$

3.2. Assume $\al\leq \be$. If $\al\geq t\be$, then
$$
dw(u,v,t) =\frac{(1+t)(2-\al)}{1+t-t\be}=\frac{9}{4}-\frac{h(\al,\be,t)}{4(1+t-t\be)},
$$
where $h(\al,\be,t):= 4(1+t)\al+t-9t\be+1$. Since,
$$
h(\al,\be,t)\geq h(t\be,\be,t)=(1-\be)(1+t)+4\be\Big(t-\frac{1}{2}\Big)^2>0,
$$
we get that $dw(u,v,t)<9/4$.

On the other hand, asume that $\al\leq t\be$. Then,
$$
dw(u,v,t)=\frac{(1+t)(2-\al)}{1+t-\al}=\frac{9}{4}-\frac{(t-\al)(1+t)+\al(2t-1)^2}{4t(t-\al+1)}.
$$
Since $\al<t$, we get that $dw(u,v,t)<9/4$.
\medskip

{\sl Case}\/ 4. Finally, assume that $u=(1,\al)$ and $v=(-1,-\be)$, with
$0\leq\al\leq 1$ and $0<\be\leq 1$. Then $\|u-v\|=\|(2,\al+\be)\|=2$ and
$\|u-tv\|=\max\{1+t,\al+t\be\}=1+t$. Therefore, $dw(u,v,t)=2<9/4$.\hfill$\Box$


\bigskip


With regards to the following example, it should be noted that the spaces $\ell_2-\ell_\infty$ and $\ell_2-\ell_1$ are dual to each other and that H. Mizuguchi, K.-S. Saito and R. Tanaka \cite{MiSaTa} showed that $DW(\ell_2-\ell_\infty)=2\sqrt{2}$. Newertheless,  it is not known if, in general, $DW(X)=DW(X^*)$.

\begin{exa}\label{l2-l1}
Let $X$ be the space $\ell_2-\ell_1$, i.e.,  $\mathbb{R}^2$ endowed with the norm
$$
\|x\|=
\begin{cases}
\;\|x\|_2 & \text{if\ \ }
x_1x_2\geq 0,\\[.5ex]
\;\|x\|_1 & \text{if\ \ } x_1x_2\leq 0,
 \end{cases}
$$
where $x=(x_1,x_2)$. Then $DW(X)=2\sqrt{2}$. Moreover,
$DW(X)>dw(u,v,t)$ for all $u,v\in S_X$ and $0<t<1$.
\end{exa}

{\it Proof.}
Let $v=(1,0)$ and for $0<t<1$, let $u_t=(t,t-1)$. Then $u_t,v\in
S_X$ and $dw(u_t,v,t)=(1+t)\sqrt{2}$. Therefore,
$DW(X)\geq\sup_{0<t<1}dw(u_t,v,t)=2\sqrt{2}$. Next, we shall see that
$dw(u,v,t)<2\sqrt{2}$ for every $u,v\in S_X$ and $0<t<1$. We can assume without
loss of generality that $u\neq v$, which implies that $\|u-tv\|\neq 0$ for
$0\leq t\leq 1$, and then all the denominators in the forthcoming counts will
be non-null. We shall consider several cases according to the position of $u$
and $v$ on $S_X$. To this aim we shall divide $S_X$ into the four arcs
\begin{align*}
S_1 & =\{(x_1,x_2)\in S_X: x_1\geq 0, x_2\geq 0\},\quad S_2=\{(x_1,x_2)\in S_X: x_1\leq 0, x_2\geq 0\},\\
S_3 & =\{(x_1,x_2)\in S_X: x_1\leq 0, x_2\leq 0\},\quad S_4=\{(x_1,x_2)\in S_X: x_1\geq 0, x_2\leq 0\},
\end{align*}
that the coordinate axes divide $S_X$. In the course of the proof we shall
identify $u,v\in S_X$ with scalars $\al,\be\in[0,1]$, and so $dw(u,v,t)$ with
$dw(\al,\be,t)$. To simplify the notation we shall consider
$\bal=\sqrt{1-\al^2}$ and $\bbe=\sqrt{1-\be^2}$.

Due to the symmetry of $S_X$ with respect to the axes $x_1=x_2$ and $x_1=-x_2$, we can limit the study to the
following six situations.

\medskip

{\sl Case}\/ 1. Assume that $u,v \in S_1$, i.e., $u=(\al,\bal)$,
$v=(\be,\bbe)$. Due to the symmetry properties, we can assume, without loss of generality, that $0\leq\be<\al\leq 1$. Then
$0\leq\bal<\bbe\leq 1$. This implies that $\al\bbe-\be\bal>0$,
$$
\|u-v\|=\|(\al-\be,\bal-\bbe)\|=\al-\be+\bbe-\bal,
$$
and
$$
\|u-tv\|=\|(\al-t\be,\bal-t\bbe)\|=
\begin{cases}
\big((\al-t\be)^2+(\bal-t\bbe)^2\big)^{1/2} & \text{ if\; } \bal-t\bbe\geq 0,\\[1ex]
\al-t\be+t\bbe-\bal & \text{ if\; } \bal-t\bbe\leq 0.
\end{cases}
$$
Therefore,
$$
dw(\al,\be,t)=\begin{cases}
\dfrac{(1+t)(\al-\be+\bbe-\bal)}{\big(1-2(\al\be+\bal\bbe)t+t^2\big)^{1/2}} &
\text{ if\; }
0<t\leq\dfrac{\bal}{\bbe},\\[4ex]
\dfrac{(1+t)(\al-\be+\bbe-\bal)}{\al-\bal+t(\bbe-\be)} & \text{ if\; }
\dfrac{\bal}{\bbe}\leq t<1,
\end{cases}
$$

\smallskip

1.1. Assume that $0<t\leq\bal/\bbe$. Since the function $t\to
1-2(\al\be+\bal\bbe)t+t^2$ is convex and attains the minimum at
$$
t_0=\al\be+\bal\bbe=\frac{\be(\al \bbe-\bal \be)+\bal}{\bbe}\geq\frac{\bal}{\bbe},
$$
we have
$$
dw(\al,\be,t)  \leq
dw\Big(\al,\be,\frac{\bal}{\bbe}\Big)=\frac{(\bal+\bbe)(\al-\be+\bbe-\bal)}{\al\bbe-\be\bal}=
1+\al\bbe+\be\bal+\bal\bbe-\al\be.
$$
Let $0\leq\theta<\phi\leq\pi/2$ be such that $\al=\cos\theta$,
$\be=\cos\phi$. Then $\bal=\sin\theta$, $\bbe=\sin\phi$ and
\begin{equation}\label{Example2.1}
\al\bbe+\be\bal+\bal\bbe-\al\be=\sin(\theta+\phi)-\cos(\theta+\phi)\leq\sqrt{2}.
\end{equation}
Therefore, $dw(\al,\be,t)\leq 1+\sqrt{2}<2\sqrt{2}$.

\smallskip

1.2. On the other hand, assume that $\bal/\bbe\leq t<1$. Then
$$
dw(\al,\be,t) = \frac{(1+t)(\al-\be+\bbe-\bal)}{\al-\bal+t(\bbe-\be)} = 1+\frac{\bbe-\be+t(\al-\bal)}{\al-\bal+t(\bbe-\be)},
$$
and we shall show that for $\bal/\bbe\leq t<1$,
$$
\frac{\bbe-\be+t(\al-\bal)}{\al-\bal+t(\bbe-\be)}\leq\sqrt{2}.
$$
The above is equivalent to seeing that the function
$$
f(t):=\bbe-\be+(\bal-\al)\sqrt{2}+t\big(\al-\bal+(\be-\bbe)\sqrt{2}\big)
$$
is negative for $\bal/\bbe\leq t<1$. Since $f(t)$ is lineal, to do this it suffices to see that $f(1)\leq 0$ and $f(\bal/\bbe)\leq 0$, which is true because
$$
f(1)=(1-\sqrt{2})(\al-\be+\bbe-\bal)<0,
$$
and, from (\ref{Example2.1}),
$$
f\left(\frac{\bal}{\bbe}\right)=\frac{(\al\bbe-\be\bal)(\al\bbe+\be\bal+\bal\bbe-\al\be-\sqrt{2})}{\bbe}\leq
0.
$$

\medskip

{\sl Case}\/ 2. Assume that $u \in S_1$ and $v\in S_2$, i.e., $u=(\al,\bal)$,
$v=(\be-1,\be)$, with $\al,\be\in[0,1]$. Then,
$$
\|u-v\|=\|(\al+1-\be,\bal-\be)\|=
\begin{cases}
\big(\al+1-\be)^2+(\bal-\be)^2\big)^{1/2} & \text{ if\; } \bal\geq\be,\\[1ex]
1+\al-\bal & \text{ if\; } \bal\leq \be,
\end{cases}
$$
and
$$
\|u-tv\|=\|(\al+t-t\be,\bal-t\be)\|=
\begin{cases}
\big((\al+t-t\be)^2+(\bal-t\be)^2\big)^{1/2} & \text{ if\; }\bal\geq t\be,\\[1ex]
t+\al-\bal & \text{ if\; }\bal\leq t\be.
\end{cases}
$$

\smallskip

2.1. Assume that $\bal\geq\be$. Then, $\bal\geq t\be$, and
$$
dw(\al,\be,t)=\frac{(1+t)\big((\al+1-\be)^2+(\bal-\be)^2\big)^{1/2}}{\big((\al+t-t\be)^2+(\bal-t\be)^2\big)^{1/2}}.
$$
Moreover,
$$
\frac{\partial dw(\al,\be,t)}{\partial
t}=\frac{\big((\al+1-\be)^2+(\bal-\be)^2\big)^{1/2}}{\big((\al+t-t\be)^2+(\bal-t\be)^2\big)^{3/2}}\,
g_1(t),
$$
where
$$
g_1(t)=1-\al+\al\be+\bal\be-t\big(1+\bal\be+(\be-1)(\al+2\be)\big).
$$
Since $g_1(t)$ is lineal, $g_1(0)\geq 0$ and $g_1(1)=2\be(1-\be)\geq 0$, it follows that
$t\in[0,1]\to dw(\al,\be,t)$ is a non-decreasing function and then
$dw(\al,\be,t)\leq dw(\al,\be,1)=2<2\sqrt{2}$.

\smallskip

2.2. Assume that $t\be\leq\bal<\be$. Then $\be>0$, and
$$
dw(\al,\be,t)=\frac{(1+t)(1+\al-\bal)}{\big((\al+t-t\be)^2+(\bal-t\be)^2\big)^{1/2}}.
$$
Moreover,
$$
\frac{\partial dw(\al,\be,t)}{\partial
t}=\frac{1+\al-\bal}{\big((\al+t-t\be)^2+(\bal-t\be)^2\big)^{3/2}}\, g_1(t),
$$
As in Case 2.1, the function $t\in[0,1]\to dw(\al,\be,t)$ is non-decreasing and then
$$
dw(\al,\be,t)\leq
dw\Big(\al,\be,\frac{\bal}{\be}\Big)=\frac{(1+\al-\bal)(\bal+\be)}{\bal(1-\be)+\al\be}.
$$
Next, we will show that $dw(\al,\be,\bal/\be)\leq 1+\sqrt{2}<2\sqrt{2}$. This is equivalent to see that
$h(\al,\be):=\al^2-1+\bal(\al-\sqrt{2})+\be(1+\bal\sqrt{2}-\sqrt{2}\al)\leq 0$.
But this follows from $h(\al,0)\leq0$ and $h(\al,1)=\al(\al+\bal-\sqrt{2})\leq
0$.

\smallskip

2.3. Assume that $\bal<t\be$. Then $\bal<t$, $\bal<\be$, and
$$
dw(\al,\be,t)=\frac{(1+t)(1+\al-\bal)}{t+\al-\bal}.
$$
Since
$$
\frac{\partial dw(\al,\be,t)}{\partial t}=\frac{-2\al\bal}{(t+\al-\bal)^2}\leq
0,
$$
we get that $dw(\al,\be,t)\leq dw(\al,\be,\bal)=1+\al+\bal\leq
1+\sqrt{2}<2\sqrt{2}$.

\medskip

{\sl Case}\/ 3. Assume that $u,v \in S_2$, i.e., $u=(\al-1,\al)$,
$v=(\be-1,\be)$, with $\al,\be\in[0,1]$. We can assume without loss of generality that
$\al>\be$, which imply $\al>t\be$. Then,
$\|u-v\|=(\al-\be)\sqrt{2}$ and
$$
\|u-tv\|=
\begin{cases}
\big((\al-1+t-t\be)^2+(\al-t\be)^2\big)^{1/2} & \text{ if }
\al-1+t-t\be\geq 0,\\[2ex]
1-t & \text{ if } \al-1+t-t\be\leq 0.
\end{cases}
$$

\smallskip

3.1. Assume that $\al-1+t-t\be\geq 0$. Then,
$$
dw(\al,\be,t)=\frac{(1+t)(\al-\be)\sqrt{2}}{\big((\al-1+t-t\be)^2+(\al-t\be)^2\big)^{1/2}}.
$$
Let us show that $dw(\al,\be,t)<2\sqrt{2}$, which is equivalent to
see that
$$
g(\al,\be,t):=(1+t)^2(\al-\be)^2-4\big((\al-1+t-t\be)^2+(\al-t\be)^2\big)<0.
$$
Since,
$$
\frac{\partial g(\al,\be,t)}{\partial \al}= 2\big((\al-\be)t^2+(2\al+6\be-4)t-7\al-\be+4\big)
$$
and
$$
\frac{\partial^2 g(\al,\be,t)}{\partial \al^2}= 2(1+t)^2-16 <0,
$$
the function $g$ is concave with respect to $\al$. Moreover, $\frac{\partial g}{\partial \al}(\al_0,\be,t)=0$ for
$$
\al_0=\frac{\be t^2+(4-6\be)t+\be-4}{t^2+2t-7}.
$$
Then,
$$
g(\al,\be,t)\leq g(\al_0,\be,t)=\frac{4(1-t)^2}{7-t^2-2t}f(\be,t),
$$
where
$$
f(\be,t)=(2\be^2-2\be+1)t^2+(4\be^2-4\be+2)t+2\be^2-2\be-3.
$$
To conclude this case, we will see that $f(\be,t)<0$, for $0\leq \be\leq 1$ and $0<t<1$. Since
$$
\frac{\partial f(\be,t)}{\partial t}=\big((2\be-1)^2+1\big)(1+t)>0,
$$
$f(\be,t)$ is strictly increasing in $t$, and then $f(\be,t)<f(\be,1)=8\be(\be-1)\leq 0$.

\smallskip

3.2. Assume that $\al-1+t-t\be\leq 0$. Then
$$
dw(\al,\be,t)=\frac{(1+t)(\al-\be)\sqrt{2}}{1-t}\leq
\frac{(1+t)(\al-\be)\sqrt{2}}{\al-t\be}\leq(1+t)\sqrt{2}<2\sqrt{2}.
$$

\medskip

{\sl Case}\/ 4. Assume that $u\in S_2$ and $v\in S_1$, i.e., $u=(\al-1,\al)$,
$v=(\be,\bbe)$. Then,
$$
\|u-v\|=
\begin{cases}
\big((1-\al+\be)^2+(\al-\bbe)^2\big)^{1/2} & \text{ if }\al\leq \bbe,\\[2ex]
1+\be-\bbe& \text{ if }\al\geq \bbe.
\end{cases}
$$
and
$$
\|u-tv\|=
\begin{cases}
\big((1-\al+t\be)^2+(\al-t\bbe)^2\big)^{1/2} & \text{ if }\al\leq t\bbe,\\[2ex]
1+t(\be-\bbe)& \text{ if }\al\geq t\bbe.
\end{cases}
$$

\smallskip

4.1. Assume that $\al\leq t\bbe$. Then $\al\leq \bbe$, and therefore
$$
dw(\al,\be,t)=\frac{(1+t)\big((1-\al+\be)^2+(\al-\bbe)^2\big)^{1/2}}{\big((1-\al+t\be)^2+(\al-t\bbe)^2\big)^{1/2}}.
$$
If $\be=1$, then $\al=0$, and $dw(0,1,t)=2<2\sqrt{2}$. Then we can assume that $0\leq\be<1$.
To see that $dw(\al,\be,t)<2\sqrt{2}$ we will show that
$f(\al,\be,t):=\big(dw(\al,\be,t)\big)^2<8$. For $0\leq\be<1$ we have
$$
\frac{\partial
f(\al,\be,t)}{\partial\be}=\frac{2\big((1-\al)\bbe+\al\be\big)(1+t)^2(1-t)\big((1-\al)^2+\al^2-t\big)}
{\bbe\big((1-\al+t\be)^2+(\al-t\bbe)^2\big)^2}.
$$
Therefore the sign of $\partial f/\partial \be$ depends on the sign of
$(1-\al)^2+\al^2-t$.

For $0<t\leq(1-\al)^2+\al^2$, the function $\be\rightarrow f(\al,\be,t)$ is non-decreasing, and then

$$
f(\al,\be,t)\leq
f(\al,1,t)=\frac{(1+t)^2\big((2-\al)^2+\al^2\big)}{(1-\al+t)^2+\al^2}.
$$
Next, we will show that $f(\al,1,t)<8$. This is equivalent to see that
$$
h_1(\al,t):=(1+t)^2\big((2-\al)^2+\al^2\big)-8\big((1-\al+t)^2+\al^2\big)<0,
$$
which is true because
\begin{align*}
h_1(\al,t) & = 2\Big(\big(\al(\al-2)-2\big)t^2+(2\al^2+4\al-4)t-7\al^2+6\al-2\Big)\\[1ex]
& < 2\big((2\al^2+4\al-4)t-7\al^2+6\al-2\big)\\[1ex]
& = 2\Big((t-1)\big(7(\al-\tfrac{3}{7})^2+\tfrac{5}{7}\big)-t\big(5(1-\al)^2+1\big)\Big)<0.
\end{align*}

Assume now that $1>t\geq(1-\al)^2+\al^2$, which implies that $\al<1$. Moreover, the function $\be\rightarrow f(\al,\be,t)$ is non-increasing, and then
$$
f(\al,\be,t)\leq
f(\al,0,t)=\frac{2(1+t)^2(1-\al)^2}{(1-\al)^2+(\al-t)^2}.
$$
Next we will show that $f(\al,0,t)<8$. This is equivalent to see that
$$
h_2(\al,t):=(1+t)^2(1-\al)^2-4\big((1-\al)^2+(\al-t)^2\big)<0.
$$
Note that
$$
h_2(\al,t)=(\al^2-2\al-3)t^2+(2\al^2+4\al+2)t-(7\al^2-6\al+3)
$$
is a second degree polynomial in $t$. Since
$$
(2\al^2+4\al+2)^2+4(\al^2-2\al-3)(7\al^2-6\al+3)=32(\al+1)(\al-1)^3<0,
$$
$h_2(\al,t)$ has complex roots in $t$, which implies that $h_2(\al,t)<0$, because $h_2(\al,0)=-7\al^2+6\al-3=-7(\al-\tfrac{3}{7})^2-\tfrac{12}{7}<0$.

\smallskip

4.2. Assume that $t\bbe<\al\leq\bbe$. Then,
$$
dw(\al,\be,t)=\frac{(1+t)\|x-y\|_2}{\|x-ty\|_1}\leq\frac{(1+t)\|x-y\|_2}{\|x-ty\|_2}<2\sqrt{2},
$$
where the last inequality follows from the calculations in the previous case (recall that there we
did not use that $\al\leq t\bbe$, only that $\al\leq\bbe$).

\smallskip

4.3. Assume that $t\bbe\leq\bbe\leq\al$. Then
$$
dw(\al,\be,t)=\frac{(1+t)(1+\be-\bbe)}{1+t(\be-\bbe)}\leq 2<2\sqrt{2},
$$
where the first inequality is equivalent to $(1-t)(1-\be+\bbe)\geq 0$.

\medskip

{\sl Case}\/ 5.  Assume that $u\in S_1$ and $v\in S_3$, i.e., $u=(\al,\bal)$,
$v=(-\be,-\bbe)$. In this case the norm of $u$, $v$, $u-v$ and $u-tv$ coincides
with the Euclidean norm. Since in inner product spaces the Dunkl--Williams
constant is equal to $2$, we get that $dw(\al,\be,t)\leq 2$.

\medskip

{\sl Case}\/ 6.  Assume that $u\in S_2$ and $v\in S_4$, i.e., $u=(\al-1,\al)$,
$v=(1-\be,-\be)$. Then $\|u-v\|=2$, and $\|u-tv\|=1+t$, from which it follows that
$dw(\al,\be,t)=2$. \hspace*{\fill} $\Box$


\bigskip

In the examples below, $DW(X)$ is calculated by using the identity (recall Proposition \ref{Equivalences})
$$
DW(X)=\sup\left\{dw(u,v): u,v\in S_X,\ u+v\neq0  \right\}.
$$
where
$$
dw(u,v):=\frac{\|u+v\|}{\inf\limits_{0<\gamma<\frac{1}{2}}\|\gamma u+(1-\gamma) v\|}.
$$
For this purpose, we use the following lemmas.

\begin{lem}\label{orthogonality}
Let $X$ be a two-dimensional normed linear space whose unit sphere $S_X$ is a polygon, and let $p^-$, $p$ and $p^+$ be three consecutive (clockwise) vertices of $S_X$.  Then $p\perp_B x$ if and only if $(p-p^-)\wedge x\geq0$ and $x\wedge(p^+-p)\geq0$, where $y\wedge z=y_1 z_2-y_2 z_1$.
\end{lem}
\begin{proof}
Just keep in mind that the wedge product determines the orientation of the vectors in the plane.
\end{proof}

\begin{lem}\label{Lemma}
If $u,v,x\in S_X$, $\ro>0$ and $\mu\in \mathbb{R}$ are such that $x\perp_B v-u$ and $\ro x=\mu u+(1-\mu) v$, then
$$
dw(u,v)\leq \frac{\|u+v\|}{|\ro|}.
$$
\end{lem}
\begin{proof}
Since $x\perp_B v-u$, we have that for all $\la\in \mathbb{R}$,
$$
\|\ro x\|\leq \|\ro x+\la(v-u)\|=\big\|(\mu-\la)u+\big(1-(\mu-\la)\big)v\big\|,
$$
and then, $|\ro|\leq\inf\limits_{0<\gamma<\frac{1}{2}}\|\gamma u+(1-\gamma) v\|$.
\end{proof}

The following example is well known. What we show here is that the supremum defining $DW(X)$ is not attained at any pair of points in $S(X)$.

\begin{exa}\label{square}
Let $X=(\mathbb{R}^2,\|\;\|_\infty)$. Then, $DW(X)=4$. Moreover, $dw(u,v)<4$ for all $u,v\in S_X$, $u+v\neq0$.
\end{exa}
\begin{proof}
We can assume that the $S_X$ is the square with vertices at the points
$$
p_0=(1,1),\quad p_1=(-1,1),\quad p_2=(-1,-1), \quad p_3=(1,-1).
$$
Let us show first that $DW_5(X)\geq 4$. Take $u_t=(1,1-2t)$, with $0<t<1$, and $v=p_1$. Then, $u_t,v\in S_X$ and $\|u_t+v\|=\|(0,2-2t)\|=2-2t$. Moreover, $\|u_t+tv\|=\|(1-t,1-t)\|=1-t$. Then
$$
DW_3(X)\geq\frac{(1+t)\|u_t+v\|}{\|u_t+tv\|}=\frac{(1+t)2(1-t)}{(1-t)}=2(1+t).
$$
Letting $t\rightarrow1$ yields $DW_3(X)\geq 4$.

Let us show now that $dw(u,v)<4$, for all $u,v\in S_X$, $u+v\neq0$. We consider several cases according to the position of $u$ and $v$ in $S_X$. We can assume without loss of generality that $u\in[p_3,p_0]$, i.e., $u=(1,1-2\al)$, with $0\leq\al\leq1$.

\medskip

{\it Case} 1. Assume that $v\in[p_3,p_0]$, i.e., $v=(1,1-2\be)$, with $0\leq\be\leq1$. Then, $\|u+v\|=2$. Moreover, for $0<\ga<\frac{1}{2}$, $\ga u+(1-\ga)v\in[p_3,p_0]$. Therefore, $dw(u,v)=2<4$.

\bigskip

{\it Case} 2. Assume that $v\in[p_0,p_1]$, i.e., $v=(2\be-1,1)$. Then, $u+v=\big(2\be,2(1-\al)\big)$, and $\|u+v\|=\max\{2\be,2(1-\al)\}$. Since $u+v\neq0$ we have that $(\al,\be)\neq(1,0)$, Moreover, we can assume that $u\neq v$, and then, $(\al,\be)\neq(0,1)$. Since, $p_0\perp_B v-u$, and $\ro p_0=\mu u+(1-\mu)v$, with
$$
\mu  = \frac{1-\be}{\al+1-\be},\quad
\ro  = \frac{(1-\al)(1-\be)+\al\be}{\al+1-\be}>0,
$$
by Lemma \ref{Lemma} we have that $dw(u,v)\leq \frac{\|u+v\|}{\ro}$,

\medskip

\noindent 2.1. Assume that $1-\al\geq\be$. Then, $\|u+v\|=2(1-\al)$, and
$$
dw(u,v)\leq \frac{\|u+v\|}{\ro}=\frac{2(1-\al)(\al+1-\be)}{(1-\al)(1-\be)+\al\be}.
$$
Showing that $\frac{\|u+v\|}{\ro}<4$ is equivalent to showing that the function
$$
g(\al,\be):=2(1-\al)(\al+1-\be)-4\big((1-\al)(1-\be)+\al\be\big)
$$
is strictly negative for $\al,\be\in[0,1]$, $(\al,\be)$ different from $(0,1)$ and $(1,0)$. This holds because $g(\al,\be)$ is linear in $\be$, $g(\al,0)=(1-\al)\big(2(1+\al)-4\big)$, and $g(\al,1)=-2\al(\al+1)$.

\medskip

\noindent 2.2. Assume that $1-\al\leq\be$. Then, $\|u+v\|=2\be$. In this case,
$$
g(\al,\be):=2\be(\al+1-\be)-4\big((1-\al)(1-\be)+\al\be\big),
$$
$g(0,\be)=2(1-\be)(\be-2)$, and $g(1,\be)=-2\be^2$.

\bigskip

{\it Case} 3. Assume that $v\in[p_1,p_2]$, i.e., $v=(-1,2\be-1)$. Then $\|u+v\|=\|(0,2(\be-\al)\|=2|\al-\be|$.

\medskip

\noindent 3.1. Assume that $\al+\be-1\geq 0$. Since,
$$
(p_0-p_3)\wedge (v-u)=4>0,\quad\text{and}\quad (v-u)\wedge (p_1-p_0)=4(\al+\be-1)\geq0,
$$
by Lemma \ref{orthogonality} we have that $p_0\perp_B v-u$. Moreover, since $\ro p_0=\mu u+(1-\mu)v$, with
$$
\mu=\frac{\be}{\al+\be},\quad \ro=\frac{\be-\al}{\al+\be},
$$
we have by Lemma \ref{Lemma} that
$$
dw(u,v)\leq \frac{\|u+v\|}{|\ro|}=2(\al+\be)<4,
$$
because $u+v\neq0$.

\medskip

\noindent 3.2. Assume that $\al+\be-1\leq 0$. In this case, $p_1\perp_B v-u$, and the proof follows as in Case 3.1.

\end{proof}

\begin{exa}\label{dodecahedron}
Let $X$ be the two-dimensional space $\mathbb{R}^2$ endowed with a norm whose unit sphere is a regular dodecahedron. Then $DW(X)=8(2-\sqrt{3})$. Moreover, $dw(u,v)<DW(X)$ for all $u,v\in S_X$, $u+v\neq0$.
\end{exa}
\begin{proof}
We will consider that $S_X$ is the regular dodecahedron with vertices
\begin{align*}
p_0 & =(1,0), & p_1 & = (\tfrac{\sqrt{3}}{2},\tfrac{1}{2}), & p_2 &= (\tfrac{1}{2},\tfrac{\sqrt{3}}{2}), & p_3 & =(0,1)\\[.5ex]
p_4 & =(-\tfrac{1}{2},\tfrac{\sqrt{3}}{2}), & p_5 & =(-\tfrac{\sqrt{3}}{2},\tfrac{1}{2}), & p_6 & =(-1,0), & p_7 & =(-\tfrac{\sqrt{3}}{2},-\tfrac{1}{2})\\[.5ex]
p_8 & = (-\tfrac{1}{2},-\tfrac{\sqrt{3}}{2}), & p_9 & =(0,-1), & p_{10} & =(\tfrac{1}{2},-\tfrac{\sqrt{3}}{2}), & p_{11} & = (\tfrac{\sqrt{3}}{2},-\tfrac{1}{2}).
\end{align*}

First, by considering the identity $DW(X)=DW_5(X)$, let us show that $DW(X)\geq 8(2-\rt)$. Let
$$
u_t=\bigg(t,\frac{1-t}{2-\rt}\bigg), \quad \frac{\sqrt{3}}{2}\leq t< 1.
$$
Then, $u_t=\al p_0+(1-\al)p_1$, with $\al=\frac{2t-\rt}{2-\sqrt{3}}\in[0,1]$, which implies $u_t\in S_X$. Let $v=p_6$. Then,
$$
u_t+v=4(1-t)\bigg(\frac{1}{2} p_3+\frac{1}{2}p_4\bigg),
$$
which implies that $\|u_t+v\|=4(1-t)$. Moreover,
$$
u_t+tv=\bigg(\frac{1-t}{2-\rt}\bigg) p_3,
$$
and then, $\|u_t+tv\|=\frac{1-t}{2-\rt}$. Therefore,
$$
DW_5(X)\geq\frac{(1+t)\|u_t+v\|}{\|u_t+tv\|}=(1+t)4(2-\rt)
$$
for all $\frac{\rt}{2}\leq t< 1$, which implies $DW(X)\geq 8(2-\rt)$.


Now, we will compute the value of $dw(u,v)$ for all $u,v\in S_X$, $u+v\neq0$. To do so, we can assume without loss of generality that $u\in[p_0,p_1]$ and $v\in[p_i,p_{i+1}]$, with $i=0,\ldots, 6$. Then,
$$
u=\al p_0+(1-\al)p_1=\bigg(\frac{(2-\sqrt{3})\al+\sqrt{3}}{2},\frac{1-\al}{2}\bigg), \quad 0\leq\al\leq 1,
$$
and $v=\be p_i+(1-\be)p_{i+1}$, $0\leq\be\leq 1$, $i=0,\ldots, 6$. We will therefore refer to $dw(u,v)$ as $dw(\al,\be)$.


As a strategy for the proof, we will frequently express the terms in the equations as a sum of positive components.

\medskip

{\it Case} 1. Let $v\in[p_0,p_1]$, i.e.,
$$
v=\be p_0+(1-\be)p_1=\bigg(\frac{(2-\sqrt{3})\be+\sqrt{3}}{2},\frac{1-\be}{2}\bigg), \quad 0\leq\be\leq 1.
$$
Since
$$
\frac{1}{2}(u+v)=\frac{\al+\be}{2}p_0+\bigg(1-\frac{\al+\be}{2}\bigg)p_1\in[p_0,p_1],
$$
we have that $\|u+v\|=2$. Since $u,v\in[p_0,p_1]$, we have that $\|\ga u+(1-\ga)v\|=1$ for $0\leq\ga\leq1$, and then $dw(\al,\be)=2<8(2-\sqrt{3})$.

\medskip

{\it Case} 2. Let $v\in[p_1,p_2]$, i.e.,
$$
v=\be p_1+(1-\be)p_2=\bigg(\frac{(\rt-1)\be+1}{2},\frac{(1-\rt)\be+\rt}{2}\bigg), \quad 0\leq\be\leq 1.
$$

\noindent 2.1. Assume that $\al+\be\geq 1$. The case $(\al,\be)=(0,1)$ is covered by Case 1. Thus, we may assume that $(\al,\be)\neq(0,1)$. Since, $u+v=\la\big(\de p_0+(1-\de)p_1\big)$, with
\begin{align*}
\de & = \frac{\al+\be-1}{\big(\al+\be-1\big)+\big(1-\al+\be+\rt(1-\be)\big)}\in[0,1],\\[.5ex]
\la & =(2-\rt)\be+\rt>0,
\end{align*}
we have that $\|u+v\|=(2-\rt)\be+\rt$. Since,
$$
(p_1-p_0)\wedge(v-u) =\frac{(2-\rt)(1-\be)}{2}\geq0
$$
and
$$
(v-u)\wedge(p_2-p_1)=\frac{(2-\rt)\al}{2}\geq 0,
$$
by Lemma \ref{orthogonality}, we have that $p_1\perp_B v-u$. Moreover, we have that $\ro p_1=\mu u+(1-\mu)v$, where
$$
\mu  =\frac{1-\be}{\al+1-\be},\quad
\ro  =\frac{f(\al,\be)}{\al+1-\be}>0,
$$
being $f(\al,\be):=(\rt-1)\al+1-\be+(2-\rt)\al\be$. Therefore, by Lemma \ref{Lemma}, we have that
%
%
%
%
%
%
$$
dw(\al,\be)\leq\frac{\|u+v\|}{\ro}=\frac{\big((2-\rt)\be+\rt\big)(\al+1-\be)}{f(\al,\be)}=:g(\al,\be)
$$
Now, let us show that $g(\al,\be)<8(2-\rt)$. Since,
$$
\frac{\partial g(\al,\be)}{\partial \al}=\frac{(2-\rt)(1-\be)^2\big(\rt(1-\be)+2\be\big)}{f(\al,\be)^2}\geq 0,
$$
it follows that $g(\al,\be)$ is monotonically increasing with respect to $\al$, and then
$$
g(\al,\be)\leq g(1,\be)=\frac{(2-\be)\big((2-\rt)\be+\rt\big)}{(1-\rt)\be+\rt}.
$$
Finally, we will see that $g(1,\be)<8(2-\rt)$. This is equivalent to showing that
$$
(\rt-2)\be^2+(21\rt-36)\be-14\rt+24<0,
$$
which is true since that polynomial has complex roots and $-14\rt+24<0$,

\bigskip

\noindent 2.2. Assume that $\al+\be\leq 1$. This case follows by symmetry from Case 2.1, with $u$ and $v$ (and $p_0$ and $p_2$) interchanged.

\bigskip

{\it Case} 3. Let $v\in[p_2,p_3]$, i.e.,
$$
v=\be p_2+(1-\be)p_3=\bigg(\frac{\be}{2},\frac{(\rt-2)\be+2}{2}\bigg), \quad 0\leq\be\leq 1.
$$
Since $u+v=\lambda\big(\de p_1+(1-\de)p_2\big)$, with
\begin{align*}
\de &= \frac{(\rt-1)\al+\be}{\big((\rt-1)\al+\be\big)+\big((1-\al)+(\rt-1)(1-\be)\big)} \in [0,1]\\[.5ex]
\la &= (2-\rt)\be+(\rt-1)\al+\rt-\al>0
\end{align*}
we have that $\|u+v\|=(2-\rt)\be+(\rt-1)\al+\rt-\al$.

\medskip

\noindent 3.1. Assume that $\al+\be\geq1$. Since,
$$
(p_1-p_0)\wedge(v-u)=\frac{(2\rt-3)(1-b)+2-\rt}{2}>0
$$
and
$$
(v-u)\wedge(p_2-p_1)=\frac{(2-\rt)(\al+\be-1)}{2}\geq0,
$$
we have that $p_1\perp_B v-u$. Moreover, we have that $\ro p_1=\mu u+(1-\mu)v$, where
$$
\mu=\frac{\be+\rt(1-\be)}{\be+\rt(1-\be)+\al},\quad \ro=\frac{f(\al,\be)}{\be+\rt(1-\be)+\al},
$$
being $f(\al,\be):=(2-\rt)\al(\rt\be+1)+\rt(1-\be)+\be>0$. By Lemma \ref{Lemma}, we have that
$$
dw(\al,\be)\leq\frac{\|u+v\|}{\ro}=\frac{\big((2-\rt)\be+(\rt-1)\al+\rt-\al\big)\big(\be+\rt(1-\be)+\al\big)}{f(\al,\be)}.
$$
Now, let us show that $\frac{\|u+v\|}{\ro}<8(2-\sqrt{3})$. To this end, we will see that the function
\begin{align*}
g(\al,\be) & := \big((2-\rt)\be+(\rt-1)\al+\rt-\al\big)\big(\be+\rt(1-\be)+\al\big)\\[.5ex]
&\phantom{:=}-8(2-\sqrt{3})f(\al,\be)\\[.5ex]
&\phantom{:}=(\rt-2)a^2+(5-3\rt)b^2+(93-54\rt)ab\\[.5ex]
&\phantom{:=}+(31\rt-53)a+(27\rt-46)b+27-16\rt
\end{align*}
is strictly negative for $\al,\be\in[0,1]$. Since $\frac{\partial^2 g(\al,\be)}{\partial\al^2}=2\rt-4<0$, we have that $g(\al,\be)$ is a concave function with respect to $\al$. Furthermore,
$\frac{\partial g(\al,\be)}{\partial \al}$ vanishes at
$$
\al(\be)=\frac{(54\rt-93)\be+53-31\rt}{2\rt-4}.
$$
Therefore, for all $\al\in \mathbb{R}$, we have that
$$
g(\al,\be)\leq g(\al(\be),\be)=\frac{4682-2703\rt}{4}{b^2}+\frac{1593\rt-2759}{2}b+\frac{817-472\rt}{2}.
$$
Finally, it is straightforward to see that the above polinomial is strictly negative for $0\leq\be\leq1$.

\medskip

\noindent 3.2. Assume that $\al+\be\leq 1$. In this case, $p_2\perp_B v-u$, because
$$
(p_2-p_1)\wedge (v-u)=\frac{(2-\rt)(1-\al-\be)}{2}\geq 0,
$$
and
$$
(v-u)\wedge (p_3-p_2)=\frac{(2\rt-3)\al+2-\rt}{2}>0.
$$
Moreover, $\ro p_2=\mu u+(1-\mu)v$, with
$$
\mu  =\frac{1-\be}{1-\be+(\rt-1)\al+1},\quad \ro=\frac{f(\al,\be)}{1-\be+(\rt-1)\al+1},
$$
where $f(\al,\be):=(2-\rt)\al(\rt\be+1)+\rt(1-\be)+\be>0$ for $\al,\be\in[0,1]$. By Lemma \ref{Lemma},
\begin{align*}
dw(\al,\be) & \leq\frac{\|u+v\|}{\ro}\\[.5ex]
& =\frac{\big((2-\rt)\be+(\rt-1)\al+\rt-\al\big)\big(1-\be+(\rt-1)\al+1\big)}{f(\al,\be)}.
\end{align*}
To show that $\frac{\|u+v\|}{\ro}<8(2-\rt)$ we will see that $g(\al,\be)<0$, for $\al,\be\in[0,1]$, where
\begin{align*}
g(\al,\be) & :=  \big((2-\rt)\be+(\rt-1)\al+\rt-\al\big)\big(1-\be+(\rt-1)\al+1\big)\\[.5ex]
&\phantom{:=\big(} -8(2-\rt)f(\al,\be)\\[.5ex]
&\phantom{:}=(2-3\rt)\al^2+(\rt-2)\be^2+(93-54\rt)\al\be\\[.5ex]
&\phantom{:=}+(33\rt-57)\al+(21\rt-36)\be+24-14\rt.
\end{align*}
In this case, since $\frac{\partial^2 f(\al,\be)}{\partial \be^2}=2\rt-4<0$, it follows that $g(\al,\be)$ is concave with respect to $\be$. Since $\frac{\partial g(\al,\be)}{\partial\be}$ vanishes at
$$
\be(\al):=\frac{(54\rt-93)\al-21\rt+36}{2\rt-4},
$$
we have that
$$
g(\al,\be) \leq g(\al,\be(\al)) = \frac{4682-2703\rt}{4}\al^2+\frac{1110\rt-1923}{2}\al+\frac{798-461\rt}{4},
$$
and it is straightforward to see that the above polynomial is strictly negative for $0\leq\al\leq1$.

\bigskip

{\it Case} 4. Let $v\in[p_3,p_4]$, i.e.,
$$
v=\be p_3+(1-\be)p_4=\bigg(\frac{\be-1}{2},\frac{(2-\rt)\be+\rt}{2}\bigg),\quad 0\leq\be\leq1.
$$

Since,
$$
(p_2-p_1)\wedge(v-u)=\frac{2-\rt}{2}(1-\al)+\frac{2\rt-3}{2}(1-b)\geq 0,
$$
and
$$
(v-u)\wedge(p_3-p_2)=\frac{2\rt-3}{2}\al+\frac{2-\rt}{2}\be\geq 0,
$$
we have that $p_2\perp_B v-u$. Moreover, $\ro p_2=\mu u+(1-\mu)v$, with
\begin{align*}
\mu & = \frac{\be+\rt(1-\be)}{\be+\rt(1-\be)+(\rt-1)\al+1},\\[1ex]
\ro & =\frac{f(\al,\be)}{\be+\rt(1-\be)+(\rt-1)\al+1},
\end{align*}
where $f(\al,\be):=(2-\rt)(2-\al-\be+2\al\be)+2(\rt-1)>0$ for $\al,\be\in[0,1]$.

\noindent 4.1. Assume now that $\al+\be\geq1$. Then, $u+v=\la\big(\de p_1+(1-\de)p_2\big)$, with
\begin{align*}
\de & = \frac{(\rt-1)(\al+\be-1)}{(\rt-1)(\al+\be-1)+(\rt-1)\be+2-\al-\be}\in[0,1],\\[.5ex]
\la &= (2-\rt)(1-\al)+(2\rt-3)\be+1>0.
\end{align*}
Therefore, $\|u+v\|=(2-\rt)(1-\al)+(2\rt-3)\be+1$. By Lemma \ref{Lemma}, we have that $dw(u,v)\leq \frac{\|u+v\|}{\ro}$. Showing that $\frac{\|u+v\|}{\ro}<8(2-\rt)$ is equivalent to showing that the function
\begin{align*}
g(\al,\be) & := \big((2-\rt)(1-\al)+(2\rt-3)\be+1\big)\\[.5ex]
& \phantom{:= \big(}\cdot\big(\be+\rt(1-\be)+(\rt-1)\al+1)\big) - 8(2-\rt)f(\al,\be)\\[.5ex]
& \phantom{:}= (5-3\rt)\al^2+(5\rt-9)\be^2+(62\rt-108)\al\be\\[.5ex]
& \phantom{:=}+(51-29\rt)\al+(65-37\rt)\be+18\rt-32
\end{align*}
is strictly negative for $\al,\be\in[0,1]$. Since $\frac{\partial^2 g(\al,\be)}{\partial \al^2}=10-6\rt<0$, we have that $g(\al,\be)$ is concave with respect to $\al$. Moreover, $\frac{\partial g(\al,\be)}{\al}$ vanishes at
$$
\al(\be)=\frac{(62\rt-108)\be-29\rt+51}{6\rt-10},
$$
and then,
\begin{align*}
g(\al,\be) & \leq g(\al(\be),\be)\\[.5ex]
& =\bigg(\frac{667\rt-1155}{2}\bigg)\be^2+(599-346\rt)\be+\frac{363\rt-629}{4}.
\end{align*}
It is straightforward to see that the above polynomial is strictly negative for $0\leq\be\leq1$.

\bigskip

\noindent 4.2. Assume that $\al+\be\leq1$. Then, $u+v=\la\big(\de p_2+(1-\de)p_3\big)$, with
\begin{align*}
\de & = \frac{(2-\rt)\al+\be+\rt-1}{(2-\rt)\al+\be+\rt-1+(\rt-1)(1-\al-\be)}\in[0,1],\\[.5ex]
\la & = (2\rt-3)(1-\al)+(2-\rt)\be+1>0.
\end{align*}
Therefore, $\|u+v\|=(2\rt-3)(1-\al)+(2-\rt)\be+1$. As in Case 4.1, showing that $\frac{\|u+v\|}{\ro}<8(2-\rt)$ is equivalent to showing that the function
\begin{align*}
g(\al,\be) & := \big((2\rt-3)(1-\al)+(2-\rt)\be+1\big)\\[.5ex]
& \phantom{:= \big(}\cdot\big(\be+\rt(1-\be)+(\rt-1)\al+1)\big) - 8(2-\rt)f(\al,\be)\\[.5ex]
& \phantom{:} = (5\rt-9)\al^2+(5-3\rt)\be^2+(62\rt-108)\al\be\\
& \phantom{:= \big(}+(61-35\rt)\al+(47-27\rt)\be+16\rt-28
\end{align*}
is strictly negative for $\al,\be\in[0,1]$. The proof follows exactly as in Case 4.1, bearing in mind that in this case $\frac{\partial^2 g(\al,\be)}{\partial \al^2}=10\rt-18<0$, $$
\al(\be)=\frac{(62\rt-108)\be-35\rt+61}{18-10\rt},
$$
and
$$
g(\al(\be),\be)=\bigg(\frac{667-385\rt}{2}\bigg)\be^2+(187\rt-324)\be+\frac{307}{4}-\frac{533\rt}{12}<0
$$
for $0\leq\be\leq1$.

\bigskip

{\it Case} 5. Let $v\in [p_4,p_5]$, i.e.,
$$
v=\be p_4+(1-\be) p_5,\quad 0\leq\be\leq1.
$$
Since, $u+v=\la\big(\de p_2+(1-\de) p_3\big)$, with
\begin{align*}
\de & =\frac{(2-\rt)\al+(\rt-1)\be}{(2-\rt)\al+(\rt-1)\be+(\rt-1)(1-\al)+(2-\rt)(1-\be)}\in[0,1],\\[.5ex]
\la & =(2\rt-3)(1-\al+\be)+2(2-\rt) >0,
\end{align*}
we have that $\|u+v\|=(2\rt-3)(1-\al+\be)+2(2-\rt)$.

\medskip

\noindent 5.1. Assume that $\al+\be\geq1$. Since,
$$
(p_2-p_1)\wedge (v-u)=\frac{(2-\rt)\big(1-\al+2(1-\be)\big)+2\rt-3}{2}>0,
$$
and
$$
(v-u)\wedge (p_3-p_2)=\frac{(2\rt-3)(\al+\be-1)}{2}\geq0,
$$
we have that $p_2\perp_B v-u$. Moreover, we have the identity $\ro p_2=\mu u+(1-\mu)v$, with
\begin{align*}
\mu & =\frac{2(1-\be)+\rt\be}{2(1-\be)+\rt\be+(\rt-1)\al+1},\\[.5ex]
\ro & =\frac{f(\al,\be)}{2(1-\be)+\rt\be+(\rt-1)\al+1},
\end{align*}
where $f(\al,\be):=(\rt-1)(1-\al)+(2-\rt)\be+(2\rt-3)\al\be+1>0$. Now, showing that $\frac{\|u+v\|}{\ro}<8(2-\rt)$ is equivalent to showing that the function
\begin{align*}
g(\al,\be) & := \big((2\rt-3)(1-\al+\be)+2(2-\rt)\big)\\[.5ex]
& \phantom{:=}\cdot\big(2(1-\be)+\rt\be+(\rt-1)\al+1\big)-8(2-\rt)f(\al,\be)\\[.5ex]
& \phantom{:}= (5\rt-9)\al^2+(12-7\rt)\be^2+(93-54\rt)\al\be\\[.5ex]
& \phantom{:=}+(19\rt-32)\al+(39\rt-67)\be+27-16\rt
\end{align*}
is strictly negative for $\al,\be\in[0,1]$. Since $\frac{\partial^2 g(\al,\be)}{\partial \al^2}=10\rt-18<0$, we have that $g(\al,\be)$ is concave with respect to $\al$. Moreover, $\frac{\partial g(\al,\be)}{\partial \al}$ vanishes at
$$
\al(\be):=\frac{(54\rt-93)\be-19\rt+32}{10\rt-18},
$$
and then,
\begin{align*}
g(\al,\be) & \leq g(\al(\be),\be)\\[.5ex]
& =\bigg(\frac{2067-1193\rt}{8}\bigg)\be^2+\bigg(\frac{551\rt-955}{4}\bigg)\be+\frac{1371-793\rt}{24}.
\end{align*}
It is straightforward to see that the above polynomial is strictly negative for $0\leq\be\leq1$.

\bigskip

\noindent 5.2. Assume that $\al+\be\leq 1$. Since,
$$
(p_3-p_2)\wedge (v-u)=\frac{(2\rt-3)\big(1-\al-\be\big)}{2}\geq0,
$$
and
$$
(v-u)\wedge (p_4-p_3)=\frac{(2-\rt)(2\al+\be+\rt)}{2}>0,
$$
we have that $p_3\perp_B v-u$. Moreover, we have the identity $\ro p_3=\mu u+(1-\mu)v$, with
\begin{align*}
\mu & = \frac{\be+\rt(1-\be)}{\be+\rt(1-\be)+(2-\rt)\al+\rt},\\[.5ex]
\ro & = \frac{f(\al,\be)}{\be+\rt(1-\be)+(2-\rt)\al+\rt},
\end{align*}
where $f(\al,\be):=(\rt-1)(1-\al)+(2-\rt)\be+(2\rt-3)\al\be+1>0$. Showing that $\frac{\|u+v\|}{\ro}<8(2-\rt)$ is equivalent to showing that the function
\begin{align*}
g(\al,\be) & := \big((2\rt-3)(1-\al+\be)+2(2-\rt)\big)\\[.5ex]
& \phantom{:=}\cdot\big(\be+\rt(1-\be)+(2-\rt)\al+\rt)-8(2-\rt)f(\al,\be)\\[.5ex]
& \phantom{:}= (12-7\rt)\al^2+(5\rt-9)\be^2+(93-54\rt)\al\be\\[.5ex]
& \phantom{:=}+(29\rt-50)\al+(25\rt-43)\be+24-14\rt
\end{align*}
is strictly negative for $\al,\be\in[0,1]$. The proof follows exactly as in Case 5.1, bearing in mind that in this case $\frac{\partial^2 g(\al,\be)}{\partial \al^2}=2(12-7\rt)<0$,
$$
\al(\be)=\frac{(54\rt-93)\be+50-29\rt}{24-14\rt},
$$
and
$$
g(\al(\be)\be)=\bigg(\frac{437\rt-756}{4}\bigg)\be^2+\bigg(\frac{301-174\rt}{2}\bigg)\be+\bigg(\frac{193\rt-336}{12}\bigg)<0
$$
for $0\leq\be\leq1$.

\bigskip

{\it Case} 6. Let $v\in[p_5,p_6]$, i.e.,
$$
v=\be p_5+(1-\be)p_6=\bigg(\frac{(2-\rt)\be-2}{2},\frac{\be}{2}\bigg),\quad 0\leq\be\leq1.
$$
Since, $u+v\neq0$, we have that $(\al,\be)\neq(1,0)$. Since,
$$
(p_3-p_2)\wedge(v-u)= \frac{(2\rt-3)(1-\al)+(4-2\rt)(1-\be)}{2}\geq0
$$
and
$$
(v-u)\wedge(p_4-p_3)=\frac{(4-2\rt)\al+(2\rt-3)\be}{2}\geq0
$$
we have that $p_3\perp_B v-u$. From the identity, $\ro p_3=\mu u+(1-\mu)v$, where
\begin{align*}
\mu & = \frac{2(1-\be)+\rt\be}{2(1-\be)+\rt\be+(2-\rt)\al+\rt},\\[.5ex]
\ro & =\frac{f(\al,\be)}{2(1-\be)+\rt\be+(2-\rt)\al+\rt},
\end{align*}
with $f(\al,\be):=1-\al+(\rt-1)\be+(2-\rt)\al\be>0$, we obtain, as in the other cases, that
\begin{equation}\label{dodecahedron_Case 6}
dw(\al,\be)\leq \frac{\|u+v\|}{\ro}=\frac{\|u+v\|\big(2(1-\be)+\rt\be+(2-\rt)\al+\rt\big)}{f(\al,\be)}.
\end{equation}
Let us now compute $\|u+v\|$.

\medskip

\noindent 6.1. Assume that $\al+\be\geq1$. Since, $u+v=\la\big(\de p_2+(1-\de)p_3\big)$, with
\begin{align*}
\de & =\frac{(2-\rt)(\al+\be-1)}{(2-\rt)(\al+\be-1)+(\rt-1)(1-\al)+(2-\rt)\be}\in[0,1],\\[.5ex]
\la & =(2-\rt)\big(\rt(1-\al)+2\be\big)>0,
\end{align*}
we have that $\|u+v\|=(2-\rt)\big(\rt(1-\al)+2\be\big)$.
From (\ref{dodecahedron_Case 6}), it follows that showing that $\frac{\|u+v\|}{\ro}<8(2-\rt)$ is equivalent to showing that the function
\begin{align*}
g(\al,\be) & := (2-\rt)\big(\rt(1-\al)+2\be\big)\\[.5ex]
&\phantom{:=}\cdot\big(2(1-\be)+\rt\be+(2-\rt)\al+\rt\big)-8(2-\rt)f(\al,\be)\\[.5ex]
&\phantom{:}= (12-7\rt)\al^2+(8\rt-14)\be^2+(31\rt-54)\al\be\\[.5ex]
&\phantom{:=}+(4-2\rt)\al+(54-31\rt)\be+9\rt-16
\end{align*}
is strictly negative for $\al,\be\in[0,1]$. Since $\frac{\partial^2 g(\al,\be)}{\partial\be^2}=16\rt-28<0$, we have that $g(\al,\be)$ is concave with respect to $\be$. Moreover, $\frac{\partial g(\al,\be)}{\partial \be}$ vanishes at
$$
\be(\al):=\frac{(\rt-6)\al+6-\rt}{4}.
$$
Therefore,
$$
g(\al,\be)\leq g(\al,\be(\al))=\bigg(\frac{1-\al}{8}\bigg)\big((296\rt-513)\al+289-168\rt\big)<0,
$$
if $0\leq\al<1$. On the other hand, if $\al=1$, then $\be\neq0$, and  $g(1,\be)=(8\rt-14)\be^2<0$.

\bigskip

\noindent{6.2.}  Assume that $\al+\be\leq1$. Since, $u+v=\la\big(\de p_3+(1-\de)p_4\big)$, with
\begin{align*}
\de & =\frac{(2-\rt)(1-\al)+(\rt-1)\be}{(2-\rt)(1-\al)+(\rt-1)\be+(2-\rt)(1-\al-\be)}\in[0,1]\\[.5ex]
\la & = (2-\rt)\big(2(1-\al)+\rt\be\big)>0,
\end{align*}
we have that $\|u+v\|=(2-\rt)\big(2(1-\al)+\rt\be\big)$. In this case we have,
\begin{align*}
g(\al,\be) & := (2-\rt)\big(2(1-\al)+\rt\be\big)\\[.5ex]
&\phantom{:=}\cdot\big(2(1-\be)+\rt\be+(2-\rt)\al+\rt\big)-8(2-\rt)f(\al,\be)\\[.5ex]
&\phantom{:}= (8\rt-14)\al^2+(12-7\rt)\be^2+(31\rt-54)\al\be\\[.5ex]
&\phantom{:=}+(28-16\rt)\al+(26-15\rt)\be+8\rt-14,
\end{align*}
Since, $\frac{\partial^2 g(\al,\be)}{\partial \al^2}=-4(2-\rt)^2<0$, we hace that $g(\al,\be)$ is concave with respect to $\al$. Moreover, $\frac{\partial g(\al,\be)}{\partial\al}$ vanishes at $\al(\be)=\frac{\rt-6}{4}\be+1$, and then, for $\al,\be\in[0,1]$,
$$
g(\al,\be)\leq g(\al(\be),\be)=\frac{\be}{8}\big((513-296\rt)\be+224-128\rt\big)<0
$$
for $0<\be\leq1$. If $\be=0$, then $\al\neq1$, and in this case we have,
$$
g(\al,0)=(8\rt-14)(1-a)^2<0
$$
for $0\leq\al<1$.

\bigskip

{\it Case} 7. Assume that $v\in[p_6,p_7]$, i.e.,
$$
v=\be p_6+(1-\be)p_7=\bigg(\frac{(\rt-2)\be-\rt}{2},\frac{\be-1}{2}\bigg),\quad 0\leq\be\leq1.
$$
Then,
$$
u+v=(4-2\rt)(\be-\al)\bigg(\frac{1}{2}p_3+\frac{1}{2}p_4\bigg).
$$
By symmetry, and having in mind that $u+v\neq0$, we can assume without loss of generality, that $\al<\be$.
Hence, $\|u+v\|=(4-2\rt)(\be-\al)$.

\medskip

\noindent 7.1. Assume that $\al+\be\geq1$. Since,
$$
(p_3-p_2)\wedge(v-u)=\frac{(2\rt-3)(2-\al-\be)+4-2\rt}{2}> 0,
$$
and
$$
(v-u)\wedge(p_4-p_3)=(2-\rt)(\al+\be-1)\geq0,
$$
we have that $p_3\perp_B v-u$.
Moreover, $\ro p_3=\mu u+(1-\mu)v$, with
\begin{align*}
\mu & =\frac{(2-\rt)\be+\rt}{(2-\rt)\be+\rt+(2-\rt)\al+\rt},\\[.5ex]
\ro & = \frac{\be-\al}{(2-\rt)\be+\rt+(2-\rt)\al+\rt}>0.
\end{align*}
Then,
\begin{align*}
dw(\al,\be) & \leq \frac{\|u+v\|}{\ro}\\[.5ex]
& =\frac{(4-2\rt)(\be-\al)\big((2-\rt)\be+\rt+(2-\rt)\al+\rt\big)}{(\be-\al)}\\[.5ex]
& = (4-2\rt)\big((2-\rt)\be+\rt+(2-\rt)\al+\rt\big)\\[.5ex]
& = 8(2-\rt)+(14-8\rt)(\al+\be-2) < 8(2-\rt),
\end{align*}
because $\al$ and $\be$ cannot be simultaneously equal to $1$.

\medskip

\noindent 7.2. Assume that $\al+\be\leq1$. Since,
$$
(p_4-p_3)\wedge(v-u)=(2-\rt)(1-\al-\be)\geq 0,
$$
and
$$
(v-u)\wedge(p_5-p_4)=\frac{(2\rt-3)(\al+\be)+4-2\rt}{2}>0,
$$
we have that $p_4\perp_B v-u$. Moreover, $\ro p_4=\mu u+(1-\mu)v$, with,
\begin{align*}
\mu & = \frac{2(1-\be)+\rt\be}{2(1-\be)+\rt\be+2(1-\al)+\rt\al},\\[.5ex]
\ro & = \frac{\be-\al}{2(1-\be)+\rt\be+2(1-\al)+\rt\al}>0.
\end{align*}
Then
\begin{align*}
dw(\al,\be)= & \leq \frac{\|u+v\|}{\ro}\\[.5ex]
 & = \frac{(4-2\rt)(\be-\al)\big(2(1-\be)+\rt\be+2(1-\al)+\rt\al\big)}{\be-\al}\\[.5ex]
 & = 8(2-\rt)-(14-8\rt)(\al+\be)<8(2-\rt),
\end{align*}
because $\al$ and $\be$ cannot be simultaneously zero.
\end{proof}

{\it Questions and remarks.} 1. The examples below show that the supremum in $DW(X)$ is attained at a pair of points only if $S_X$ is affine to a regular hexagon. Can this be related to the fact that, in such a case, the norm is a Radon norm?

2. In all the examples we know, $DW(X)=DW(X^*)$. Is this true in general?


\bigskip

\noindent Javier Alonso,

\noindent University of Extremadura (Retired),

\noindent 06071 Badajoz (Spain),

\noindent jalonsoro@gmail.com
\medskip

\noindent Pedro Mart\'{\i}n,

\noindent University of Extremadura,

\noindent 06071 Badajoz (Spain),

\noindent pjimenez@unex.es

\end{document}